\documentclass[12pt]{amsart}
\usepackage{amsmath,amssymb,amsbsy,amsfonts,latexsym,amsopn,amstext,cite,
                                               amsxtra,euscript,amscd,bm}
\usepackage{url}
\usepackage{mathrsfs}

\usepackage{color}
\usepackage[colorlinks,linkcolor=blue,anchorcolor=blue,citecolor=blue,backref=page]{hyperref}
\usepackage{color}
\usepackage{graphics,epsfig}
\usepackage{graphicx}
\usepackage{float}
\usepackage{epstopdf}
\hypersetup{breaklinks=true}

\usepackage[top=1in, bottom=1in, left=1in, right=1in]{geometry}

\usepackage[np]{numprint}
\npdecimalsign{\ensuremath{.}}

\def\Th1{\varTheta}

\usepackage{bibentry}
\usepackage{bbm}

\usepackage[english]{babel}
\usepackage{mathtools}
\usepackage{todonotes}
\usepackage{url}
\usepackage[colorlinks,linkcolor=blue,anchorcolor=blue,citecolor=blue,backref=page]{hyperref}

\usepackage[norefs,nocites]{refcheck}

\begin{document}

\newtheorem{theorem}{Theorem}
\newtheorem{lemma}[theorem]{Lemma}
\newtheorem{claim}[theorem]{Claim}
\newtheorem{cor}[theorem]{Corollary}
\newtheorem{conj}[theorem]{Conjecture}
\newtheorem{prop}[theorem]{Proposition}
\newtheorem{definition}[theorem]{Definition}
\newtheorem{question}[theorem]{Question}
\newtheorem{example}[theorem]{Example}
\newcommand{\hh}{{{\mathrm h}}}
\newtheorem{remark}[theorem]{Remark}

\numberwithin{equation}{section}
\numberwithin{theorem}{section}
\numberwithin{table}{section}
\numberwithin{figure}{section}

\def\sssum{\mathop{\sum\!\sum\!\sum}}
\def\ssum{\mathop{\sum\ldots \sum}}
\def\iint{\mathop{\int\ldots \int}}

\newcommand{\diam}{\operatorname{diam}}

\def\squareforqed{\hbox{\rlap{$\sqcap$}$\sqcup$}}
\def\qed{\ifmmode\squareforqed\else{\unskip\nobreak\hfil
\penalty50\hskip1em \nobreak\hfil\squareforqed
\parfillskip=0pt\finalhyphendemerits=0\endgraf}\fi}

\newfont{\teneufm}{eufm10}
\newfont{\seveneufm}{eufm7}
\newfont{\fiveeufm}{eufm5}
%
%
\newfam\eufmfam
     \textfont\eufmfam=\teneufm
\scriptfont\eufmfam=\seveneufm
     \scriptscriptfont\eufmfam=\fiveeufm
%
%
\def\frak#1{{\fam\eufmfam\relax#1}}

\newcommand{\bflambda}{{\boldsymbol{\lambda}}}
\newcommand{\bfmu}{{\boldsymbol{\mu}}}
\newcommand{\bfxi}{{\boldsymbol{\eta}}}
\newcommand{\bfrho}{{\boldsymbol{\rho}}}

\def\eps{\varepsilon}

\def\fK{\mathfrak K}
\def\fT{\mathfrak{T}}
\def\fL{\mathfrak L}
\def\fR{\mathfrak R}

\def\fA{{\mathfrak A}}
\def\fB{{\mathfrak B}}
\def\fC{{\mathfrak C}}
\def\fM{{\mathfrak M}}
\def\fS{{\mathfrak  S}}
\def\fU{{\mathfrak U}}
\def\fW{{\mathfrak W}}

\def\T {\mathsf {T}}
\def\Tor{\mathsf{T}_d}
\def\Tore{\widetilde{\mathrm{T}}_{d} }

\def\sM {\mathsf {M}}

\def\ss{\mathsf {s}}

\def\Kmnd{\cK_d(m,n)}
\def\Kmnp{\cK_p(m,n)}
\def\Kmnq{\cK_q(m,n)}

\def \balpha{\bm{\alpha}}
\def \bbeta{\bm{\beta}}
\def \bgamma{\bm{\gamma}}
\def \bdelta{\bm{\delta}}
\def \bzeta{\bm{\zeta}}
\def \blambda{\bm{\lambda}}
\def \bchi{\bm{\chi}}
\def \bphi{\bm{\varphi}}
\def \bpsi{\bm{\psi}}
\def \bnu{\bm{\nu}}
\def \bomega{\bm{\omega}}

\def \bell{\bm{\ell}}

\def\eqref#1{(\ref{#1})}

\def\vec#1{\mathbf{#1}}

\newcommand{\abs}[1]{\left| #1 \right|}

\def\Zq{\mathbb{Z}_q}
\def\Zqx{\mathbb{Z}_q^*}
\def\Zd{\mathbb{Z}_d}
\def\Zdx{\mathbb{Z}_d^*}
\def\Zf{\mathbb{Z}_f}
\def\Zfx{\mathbb{Z}_f^*}
\def\Zp{\mathbb{Z}_p}
\def\Zpx{\mathbb{Z}_p^*}
\def\cM{\mathcal M}
\def\cE{\mathcal E}
\def\cH{\mathcal H}

\def\le{\leqslant}

\def\ge{\geqslant}

\def\sfB{\mathsf {B}}
\def\sfC{\mathsf {C}}
\def\L{\mathsf {L}}
\def\FF{\mathsf {F}}

\def\sE {\mathscr{E}}
\def\sS {\mathscr{S}}

\def\cA{{\mathcal A}}
\def\cB{{\mathcal B}}
\def\cC{{\mathcal C}}
\def\cD{{\mathcal D}}
\def\cE{{\mathcal E}}
\def\cF{{\mathcal F}}
\def\cG{{\mathcal G}}
\def\cH{{\mathcal H}}
\def\cI{{\mathcal I}}
\def\cJ{{\mathcal J}}
\def\cK{{\mathcal K}}
\def\cL{{\mathcal L}}
\def\cM{{\mathcal M}}
\def\cN{{\mathcal N}}
\def\cO{{\mathcal O}}
\def\cP{{\mathcal P}}
\def\cQ{{\mathcal Q}}
\def\cR{{\mathcal R}}
\def\cS{{\mathcal S}}
\def\cT{{\mathcal T}}
\def\cU{{\mathcal U}}
\def\cV{{\mathcal V}}
\def\cW{{\mathcal W}}
\def\cX{{\mathcal X}}
\def\cY{{\mathcal Y}}
\def\cZ{{\mathcal Z}}
\newcommand{\rmod}[1]{\: \mbox{mod} \: #1}

\def\cg{{\mathcal g}}

\def\vy{\mathbf y}
\def\vr{\mathbf r}
\def\vx{\mathbf x}
\def\va{\mathbf a}
\def\vb{\mathbf b}
\def\vc{\mathbf c}
\def\ve{\mathbf e}
\def\vh{\mathbf h}
\def\vk{\mathbf k}
\def\vm{\mathbf m}
\def\vz{\mathbf z}
\def\vu{\mathbf u}
\def\vv{\mathbf v}

\def\e{{\mathbf{\,e}}}
\def\ep{{\mathbf{\,e}}_p}
\def\eq{{\mathbf{\,e}}_q}

\def\Tr{{\mathrm{Tr}}}
\def\Nm{{\mathrm{Nm}}}

 \def\SS{{\mathbf{S}}}

\def\lcm{{\mathrm{lcm}}}

 \def\0{{\mathbf{0}}}

\def\({\left(}
\def\){\right)}
\def\l|{\left|}
\def\r|{\right|}
\def\fl#1{\left\lfloor#1\right\rfloor}
\def\rf#1{\left\lceil#1\right\rceil}
\def\fl#1{\left\lfloor#1\right\rfloor}
\def\ni#1{\left\lfloor#1\right\rceil}
\def\sumstar#1{\mathop{\sum\vphantom|^{\!\!*}\,}_{#1}}

\def\mand{\qquad \mbox{and} \qquad}

\def\tblue#1{\begin{color}{blue}{{#1}}\end{color}}




\hyphenation{re-pub-lished}

\mathsurround=1pt

\def\bfdefault{b}

\def \F{{\mathbb F}}
\def \K{{\mathbb K}}
\def \N{{\mathbb N}}
\def \Z{{\mathbb Z}}
\def \P{{\mathbb P}}
\def \Q{{\mathbb Q}}
\def \R{{\mathbb R}}
\def \C{{\mathbb C}}
\def\Fp{\F_p}
\def \fp{\Fp^*}

 \def \xbar{\overline x}

\title[Polynomial Corners Over Finite Fields]
{Polynomial Corners Over Finite Fields}

\author [L. P. Wijaya]{Laurence P. Wijaya}
\address{Department of Mathematics, University of Kentucky, 715 Patterson Office Tower, Lexington, KY 40506, USA}
\email{laurence.wijaya@uky.edu}

\begin{abstract}  Recently there has been some progress in understanding the density of a subset of $[N]^2$ that avoids polynomial patterns. Kravitz, Kuca, and Leng showed that if $P\in\Z[z]$ satisfies certain conditions, then any set $A\subseteq[N]^2$ does not contain $(x,y),(x+P(z),y),(x,y+P(z))$, we must have 
 \[
 |A|\ll_P\frac{N^2}{(\log\log\log N)^c}
 \]
 for some small constant $c$. 
 
 In this article, we show a similar result in $(\F_p)^2$ where we get a better bound on the density of a set $A\subseteq (\F_p)^2$ not containing $(x,y),(x+P(z),y),(x,y+P(z))$ with some conditions on $P\in \F_p[z]$.
 \end{abstract}


\maketitle


\section{Introduction}
Roth's theorem in arithmetic progression started the study of patterns in dense subsets of the set of natural numbers. Roth showed that every sufficiently dense subset of the set of natural numbers contains a non-trivial arithmetic progression of length $3$, that is, a pattern of the form $x,x+d,x+2d$ where both $x,d$ are positive integers. On the other hand, one can study a similar phenomenon in higher dimensions. In 1974, Ajtai and Szemer\'{e}di \cite{AS} showed that every sufficiently dense subset of $\N^2$ contains a pattern of the form $(x,y),(x+d,y),(x,y+d)$ with $x,y,d\in\N$, which we call \textit{linear corner}.

Around thirty years ago, Bergelson and Leibman \cite{BL} proved the following statement : if $\textbf{P}_1,\dots,\textbf{P}_k\in\Z^n[z]$ satisfying certain conditions, then every subset of $\Z^n$ with positive upper Banach density must contain a non-trivial pattern of the form
\[
\textbf{x},\textbf{x}+\textbf{P}_1(z),\dots,\textbf{x}+\textbf{P}_k(z).
\]
As a special case, this means that every subset of $\N^2$ with upper positive Banach density contains a pattern of the form $(x,y),(x+P_1(z),y),(x,y+P_2(z))$ with $x,y,z\in\N$ for $P_1,P_2\in \N[z]$ satisfying certain conditions. In recent years, there has been some interest in obtaining a quantitative version of this theorem. The best estimate for the case $P_1=P_2$ is due to Kravitz, Kuca, and Leng \cite{KKL}. In fact, by combining their proof with \cite[Theorem 1.1]{JLLOS}, which we mention again as Theorem \ref{lem:lincor}, we have the following improvement.
\begin{theorem}
    Let $P\in\Z[z]$ be a polynomial with an integer root of multiplicity $1$. If $A\subseteq [N]^2$ where $[N]:=\{1,\dots,N\}$ does not contain a configuration of the form $(x,y),(x+P(z),y),(x,y+P(z))$ with $P(z)\neq 0$, then
    \[
    |A|\ll_P\frac{N^2}{\exp((\log\log N)^{c_P})}
    \]
    for some constant $c_P>0$ depending only on the polynomial $P$.
\end{theorem}
In the case where $P_1\neq P_2$ not much is known. The only known case is due to Peluse, Prendiville, and Shao \cite{PPS}. They proved that if $A\subseteq [N]^2$ does not contain a configuration of the form $(x,y),(x+z,y),(x,y+z^2)$ with $z\neq 0$, then $|A|\ll \frac{N^2}{(\log N)^c}$ for some absolute constant $c>0$.

In this article, we consider the quantitative problem over $\F_p$ where $p$ is prime. It turns out that we can get a better result compared to the integer case. In fact, the known density of a set that does not contain linear corner is the same as the one that does not contain a polynomial corner. We state here our main result.
\begin{theorem}
\label{thm:corner}
    Let $p$ be a sufficiently large prime such that $p>d^2$. Let $P\in \F_p[z]$ be a polynomial of degree $d\geq 1$.

    If $A\subseteq (\F_p)^2$ does not contain a pattern of the form $(x,y),(x+P(z),y),(x,y+P(z))$, then we have
    \[
    |A|\ll_P\frac{p^2}{\exp(c_P\log p)^{1/600}}
    \]
    for some $c_P>0$.
\end{theorem}
Unlike the case of integers, the image of $P$ in $\F_p$ is quite dense, so this gives no local obstruction when we consider our problem. We also remark that in the case of one dimensional polynomial progressions, i.e., polynomial progressions in $[N]$, some results can be found at \cite{AS2,KKL}. For polynomial progressions in finite fields, Peluse \cite{Pel} found the best bound under the condition that the values of the polynomials at $0$ are $0$. In that paper, she introduced a technique that is widely used now, called the \textit{degree-lowering method}.

We briefly outline the key ideas behind our main result. We follow closely the method from \cite{KKL}, but we have some simplifications here. We consider the difference between counting operators of polynomial corners in $A$ and linear corners, and this introduces a weight to the polynomial corners counting operator. This is called the \textit{transference principle}. In our setting, the associated weight is bounded, unlike the case of the integer case where the associated weights are unbounded. After several applications of Cauchy-Schwarz inequality, we end up with bounding the Gowers $U^3(\F_p)$ norm of our weight. It turns out that by the nature of our weight and the fact that we are working in finite fields, we do not need to invoke the Inverse Theorem of Gowers Norm. Instead, we use some algebraic geometry tools to establish a bound for the $U^3(\F_p)$ norm of our weight, and this simplifies calculations.

\section{Notations and Conventions}
We use the notation $f\ll g$ or $g\gg f$ or $f=O(g)$ as being equivalent to the statement $|f|\leq Cg$ for some constant $C>0$. If the constant depends on a parameter, for example if it depends on $d,$ we use the subscript $f=O_{d}(g)$ or $f\ll_d g$ or $g\gg_d f$ to emphasize that $C=C(d)$ depends only on $d.$ We write $f=o(g)$ if $(f/g)(x)$ tends to zero as $x$ goes to infinity. We again write the subscript if the decay depends on some parameters.

For any function $f : \F_p\rightarrow \C$, the \textit{$k$-th Gowers norm} is defined by
\[
\|f\|_{U^k(\F_p)}:=\left(\underset{x,h_1,\dots,h_k\in\F_p}{\mathbb{E}} \prod_{\omega\in\{0,1\}^k} \mathcal{C}^{|\omega|}f(x+\omega\cdot (h_1,\dots,h_k))\right)^{1/2^k}.
\]
Here, we have that if $\omega=(\omega_1,\dots,\omega_k)$, the notation $|\omega|$ means $\omega_1+\dots+\omega_k$ and $\omega\cdot (h_1,\dots,h_k)=\omega_1h_1+\dots+\omega_k h_k$, also $\mathcal{C}$ is the conjugation operator.
We just write $U^k$ instead of $U^k(\F_p)$ for the $k$-th Gowers norm since the underlying space will remain fixed.

We write $\Delta_hf(x)$ for $f(x)\overline{f(x+h)}$ where $\overline{f}$ is the complex conjugate of $f$. Here $\mathbb{E}$ denotes the average or expectation operator, that is,
    \[
    \underset{x\in\F_p}{\mathbb{E}}f(x)=\frac{1}{p}\sum_{x\in\F_p}f(x)
    \]
    and we define inductively
    \[
    \underset{h_1,\dots,h_k\in\F_p}{\mathbb{E}}=\underset{h_1\in\F_p}{\mathbb{E}}\left(\underset{h_2,\dots,h_k\in\F_p}{\mathbb{E}}\right).
    \]
    So using these notations, we get that
\[
\|f\|_{U^k}:=\left(\underset{x,h_1,\dots,h_k\in\F_p}{\mathbb{E}} \Delta_{h_1}\dots\Delta_{h_k}f(x)\right)^{1/2^k}.
\]
We define the Fourier transform of $f$ at $r\in \F_p$ by
\[
\widehat{f}(r):=\frac{1}{p}\sum_{x\in \F_p}f(x)e_p(-rx)
\]
 where $e_p(x):=e^{2\pi ix/p}.$ For $1\leq q<\infty$, we define the $\ell^q-$ norm for $\widehat{f}$ by
 \[
 \|\widehat{f}\|_q:=\left(\sum_{r\in\F_p}|\widehat{f}(r)|^q\right)^{1/q}.
 \]
 We denote $1_{P(\F_p)}$ to be the characteristic function of the image of $\F_p$ under the polynomial $P$. For our main analysis, we denote the function $C_P1_{P(\F_p)}(z)$ by $w(z)$ where $C_P$ is the unique constant such that
\[
\underset{z\in\F_p}{\mathbb{E}}w(z)=1.
\]
When $A$ is a finite set, we have $\displaystyle\F_p^A:=\prod_{a\in A}\F_p$. We also denote the algebraic closure of $\F_p$ by $\overline{\F}_p$. For a variety $X$, the collection of points of $X$ in $\F_p$ is denoted by $X(\F_p).$

\section{Collection of Estimates}
\label{sec:prel}
We record several known results that will be useful to prove Theorem \ref{thm:corner}. We mention two results that come from algebraic geometry, the first one is due to Weil as a consequence of his proof of the Riemann Hypothesis for curves over finite fields. An example of modern treatment of this can be found at \cite{IK}.
\begin{prop}
\label{lem:Weil}
    Let $f$ be a non-constant polynomial of degree $d$ over $\F_p$. Then we have, for nonzero $\theta\in\F_p$
    \[
    \left|\sum_{x\in \F_p} e(\theta f(x))\right|\leq (d-1)\sqrt{p}.
    \]
    In particular, the bound is better compared to the trivial bound if $p\gg d^2$.
\end{prop}
Besides the previous lemma, we also need the following estimate that again comes from the study of the Riemann Hypothesis over finite fields. It is known as Lang-Weil estimate \cite{LW}.
\begin{prop}
\label{lem:LangWeil}
    Let $X$ be an affine variety over $\F_p$ of dimension $n$, i.e., $n$ is the maximal length of chains of distinct nonempty irreducible subvarieties $X_0\subseteq X_1\subseteq \dots\subseteq X_n$ of $X$. Let $C$ be the number of irreducible components of $X$ of dimension $n$ such that these components are also irreducible when we consider $X$ as variety over $\overline{\F}_p$. Such a component is called absolutely or geometrically irreducible. Then we have
    \[
    |X(\F_p)|=C\cdot p^n+O(p^{n-1/2})
    \]
    where $|X(\F_p)|$ is the number of points of $X$ at $\F_p$. The constant at the big-$O$ notation depends only on the geometry of $X_{\overline{\F_p}}$, the variety when we consider it as variety over $\overline{\F}_p$.
\end{prop}
We also need the bound that comes from \cite{JLLOS} for the cardinality of a set that does not contain nontrivial linear corners.
\begin{theorem}
\label{lem:lincor}
    There exists a constant $c>0$ such that the following holds. Let $(G,+)$ be an abelian group. Let $A\subseteq G\times G$ with no $x,y,d\in G,d\neq 0$ such that $(x,y),(x+d,y),(x,y+d)\in A$. Then we have
    \[
    |A|\leq\frac{|G|^2}{\exp(c(\log |G|)^{1/600})}
    \]
\end{theorem}

We show here the main claim that will lead to our main theorem. Let $f_0,f_1,f_2 : (\F_p)^2\rightarrow \C$ be functions such that $|f_i|\leq 1$ for any $i=0,1,2$. We call these functions to be \textit{$1$-bounded}. We want to consider the difference between
\[
\Lambda(f_0,f_1,f_2)=\underset{x,y,z\in\F_p}{\mathbb{E}}f_0(x,y)f_1(x+P(z),y)f_2(x,y+P(z))
\]
and
\[
\Lambda^{Model}(f_0,f_1,f_2)=\underset{x,y,z\in\F_p}{\mathbb{E}} f_0(x,y)f_1(x+z,y)f_2(x,y+z).
\]
Our goal is to show that this difference is small enough to allow us to use the result from the linear corners to deduce the density of the set that does not contain polynomial corners in $(\F_p)^2$. Turns out that we can bound the difference by $U^3$ norm of an appropriate function. We prove a more general statement.
\begin{lemma}
\label{prop:weight}
    Let $g : \F_p\rightarrow \C$ be an arbitrary function, $f_1,\dots,f_k : \F_p^D\rightarrow \C$ be $1$-bounded functions for some positive integers $k$ and $D$. We also let $\textbf{v}_1,\dots,\textbf{v}_k\in\F_p^D$ be any vectors in $\F_p^D$. Then we have
    \[
    \left|\underset{\substack{\textbf{x},\textbf{v}_1,\dots,\textbf{v}_k\in \F_p^D\\z\in\F_p}}{\mathbb{E}} g(z)f_0(\textbf{x})f_1(\textbf{x}+z\textbf{v}_1)\dots f_k(\textbf{x}+z\textbf{v}_k)\right|\leq \|g\|_{U^{k+1}}.
    \]
\end{lemma}
\begin{proof}
    We apply the Cauchy-Schwarz inequality repeatedly to obtain the result. Use Cauchy-Schwarz to double the variable $z$, we have that
    \begin{align*}
    &\left|\underset{\substack{\textbf{x},\textbf{v}_1,\dots,\textbf{v}_k\in \F_p^D\\z\in\F_p}}{\mathbb{E}} g(z)f_0(\textbf{x})f_1(\textbf{x}+z\textbf{v}_1)\dots f_k(\textbf{x}+z\textbf{v}_k)\right|\\
    &\leq\left|\underset{\substack{\textbf{x},\textbf{v}_1,\dots,\textbf{v}_k\in \F_p^D\\z,h_1\in\F_p}}{\mathbb{E}} \Delta_{h_1}g(z)\Delta_{h_1\textbf{v}_1}f_1(\textbf{x}+z\textbf{v}_1)\dots \Delta_{h_1\textbf{v}_k}f_k(\textbf{x}+z\textbf{v}_k)\right|^{1/2}.
    \end{align*}
    Next we shift $\textbf{x}$ to $\textbf{x}-z\textbf{v}_1$ and apply Cauchy-Schwarz again to find that
    \begin{align*}
        &\left|\underset{\substack{\textbf{x},\textbf{v}_1,\dots,\textbf{v}_k\in \F_p^D\\z,h_1\in\F_p}}{\mathbb{E}} \Delta_{h_1}g(z)\Delta_{h_1\textbf{v}_1}f_1(\textbf{x}+z\textbf{v}_1)\dots \Delta_{h_1\textbf{v}_k}f_k(\textbf{x}+z\textbf{v}_k)\right|\\
        &\leq\left|\underset{\substack{\textbf{x},\textbf{v}_1,\dots,\textbf{v}_k\in \F_p^D\\z,h_1,h_2\in\F_p}}{\mathbb{E}} \Delta_{h_2}\Delta_{h_1}g(z)\Delta_{h_2(\textbf{v}_2-\textbf{v}_1)}\Delta_{h_1\textbf{v}_2}f_2(\textbf{x}+z(\textbf{v}_2-\textbf{v}_1))\dots \Delta_{h_2(\textbf{v}_k-\textbf{v}_1)}\Delta_{h_1\textbf{v}_k}f_k(\textbf{x}+z(\textbf{v}_k-\textbf{v}_1))\right|^{1/2}.
    \end{align*}
    Repeat the process until we are left with only $g$, we end up with
    \begin{align*}
     &\left|\underset{\substack{\textbf{x},\textbf{v}_1,\dots,\textbf{v}_k\in \F_p^D\\z\in\F_p}}{\mathbb{E}} g(z)f_0(\textbf{x})f_1(\textbf{x}+z\textbf{v}_1)\dots f_k(\textbf{x}+z\textbf{v}_k)\right|\\
     &\leq \left|\underset{z,h_1,\dots,h_s,h_{s+1}\in\F_p}{\mathbb{E}} \Delta_{h_{k+1}}\Delta_{h_k}\dots\Delta_{h_1}g(z)\right|^{1/2^{k+1}}=\|g\|_{U^{k+1}}.
    \end{align*}
\end{proof}
\section{Proof of Main Theorem}
In this section we prove our main theorem provided we have the following results.
\begin{prop}
\label{prop:gowers}
    Let $p$ be a sufficiently large prime such that $p>d^2$. Let $P\in \F_p[z]$ be a polynomial of degree $d\geq 1$. Recall that $w(x)=C_P1_{P(\F_p)}(x)$ where $C_P$ is the unique constant such that $\displaystyle \sum_{x\in \F_p}w(x)=p$.
    \begin{enumerate}
        \item We have \[
    \|w-1\|_{U^2}=O_d(p^{-1/4}).
    \]
    \item For $k\geq 3,$, we have
    \[
    \|w-1\|_{U^k}=O_{d,k}(p^{-1/2^{k+1}}).
    \]
    \end{enumerate}
\end{prop}
\begin{remark}
    The method used to prove the bound $\|w-1\|_{U^k}=O_{d,k}(p^{-1/2^{k+1}})$ for $k\geq 3$ also works for $k=2$ and yields the same result. The Fourier-analytic proof mentioned below gives a stronger estimate as seen in the statement of the Proposition \ref{prop:gowers}.
\end{remark}

We proceed with the proof of Theorem \ref{thm:corner}.
\begin{proof}[Proof of Theorem \ref{thm:corner}]
We utilize the transference principle. Let $A\subseteq (\F_p)^2$ be a subset of density $\alpha$ such that $A$ contains a configuration of the form $(x,y),(x+P(z),y),(x,y+P(z))$ with $P$ is a polynomial of degree $d\geq 3$. By Lemma \ref{prop:weight}, it follows that
\[
|\Lambda(1_A,1_A,1_A)-\Lambda^{Model}(1_A,1_A,1_A)|\leq \|w-1\|_{U^3}.
\]
Applying Proposition \ref{prop:gowers} with $k=3$, we obtain
\[
\Lambda(1_A,1_A,1_A)=\Lambda^{Model}(1_A,1_A,1_A)+O_d(p^{-1/16}).
\]
Therefore, using the previous equality alongside a supersaturation result of Theorem \ref{lem:lincor} (see for example Appendix F of \cite{KKL}), we conclude that
\[
\gamma\gg \frac{1}{\exp((c_P\log p)^{1/600})}.
\]    
\end{proof}
\begin{remark}
    We can generalize this transference principle to arbitrary linear patterns. This means that if $A'\subseteq (\F_p)^D$ and $\textbf{v}_1,\dots,\textbf{v}_k\in\F_p^D$ be any vectors in $\F_p^D$, then by using Lemma \ref{prop:weight} and Proposition \ref{prop:gowers}, we get
    \[
|\Lambda(1_{A'},1_{A'},\dots,1_{A'})-\Lambda^{Model}(1_{A'},1_{A'},\dots,1_{A'})|\leq \|w-1\|_{U^{k+1}},
\]
where
\[
\Lambda(f_0,\dots,f_k)=\underset{\substack{\textbf{x},\textbf{v}_1,\dots,\textbf{v}_k\in \F_p^D\\z\in\F_p}}{\mathbb{E}} f_0(\textbf{x})f_1(\textbf{x}+z\textbf{v}_1)\dots f_k(\textbf{x}+z\textbf{v}_k)
\]
and
\[
\Lambda^{Model}(f_0,\dots,f_k)=\underset{\substack{\textbf{x},\textbf{v}_1,\dots,\textbf{v}_k\in \F_p^D\\z\in\F_p}}{\mathbb{E}} f_0(\textbf{x})f_1(\textbf{x}+P(z)\textbf{v}_1)\dots f_k(\textbf{x}+P(z)\textbf{v}_k)
\]
for any $1$-bounded functions $f_0,\dots,f_k$.

This implies that as long as the density of the set $A'$ that does not contain the specified linear pattern $\textbf{x},\textbf{x}+z\textbf{v}_1,\dots,\textbf{x}+z\textbf{v}_k$, is not polynomial in $p$ (some negative power of $p$), then $A'$ also does not contain the polynomial pattern $\textbf{x},\textbf{x}+P(z)\textbf{v}_1,\dots,\textbf{x}+P(z)\textbf{v}_k$ where the degree of the polynomial $P$ is less than $p^{1/2}.$

\end{remark}
It now remains to show the Proposition \ref{prop:gowers}.
\begin{proof}[Proof of Proposition \ref{prop:gowers}]
    To simplify our notation, let $S:=P(\F_p)$. Now instead of invoking the Inverse Theorem of Gowers Norm, we bound $\|w-1\|_{U^k}$ for $k\geq 3$ directly from the definition and using some algebraic geometry tools. For $k=2$ we use Lemma \ref{lem:Weil} instead. We proceed with the case $k=2$ first. It is well-known that for any function $f : \F_p\rightarrow \C,$
    \[
    \|f\|_{U^2}^2=\|\widehat{f}\|_{4}^2.
    \]
    If $f=w-1$, then by Lemma \ref{lem:Weil}
    \[
    \widehat{f}(0)=0
    \]
    and
    \[
    \widehat{f}(r)=\frac{1}{p}\sum_{x\in\F_p}f(x)e_p(-rP(x))=\frac{1}{|S|}\sum_{x\in \F_p}e_p(-rP(x))\ll_d p^{-1/2},\quad r\neq 0,
    \]
    which yields the desired bound for $k=2$.
    
    For $k\geq 3$, we are going to use the algebraic geometry tools that are mentioned in Section \ref{sec:prel}. By definition, using also the fact that $w-1$ is real-valued, we have that
    \begin{align*}
        \|w-1\|_{U^k}^{2^k}
        =\underset{x,h_1,\dots,h_k\in \F_p}{\mathbb{E}}\prod_{\omega\in\{0,1\}^k} (w-1)(x+\omega\cdot(h_1,\dots,h_k)).
    \end{align*}
    Now we denote $f=w-1$ and
    \[
    N(x):=|\{y\in \F_p : P(y)=x\}|.
    \] 
    By Lagrange interpolation, there exists a polynomial $Q(t)=\sum_{j=0}^d a_jt^j\in\Q[t]$ such that $Q(0)=0, Q(m)=1$ for $m=1,\dots,d$. Since $N(x)\in\{0,\dots,d\}$, we have
    \[
    1_S(x)=Q(N(x))
    \]
    for any $x\in\F_p$. Note that the coefficients $a_j$ depend only on $d$. If we denote by
    \[
    \mu_j=\frac{1}{p}\sum_{x\in\F_p}N(x)^j,
    \]
    we get the fact that $\displaystyle \sum_{j=1}^da_jC_P \mu_j=1$, so
    \[
    f(x)=\sum_{j=1}^d a_jC_P(N(x)^j-\mu_j).
    \]
    Using the fact that the $U^k$ is norm for any $k\geq 3$, we have
    \[
    \|f\|_{U^k}\leq C_d\left(\max_{1\leq j\leq d} \|N^j-\mu_j\|_{U^k}\right)
    \]
    for some constant $C_d$ that depends only on $d$, since we have that $C_P$ and $a_j,j=1,\dots,d$ depend only on $d$ and are bounded. Now we shift our focus to bound $\|N^j-\mu_j\|_{U^k}$ and we claim that this quantity is $O_{d,k}(p^{-1/2^{k+1}})$.

    Fix $j\in\{1,\dots,d\}$ and let $g=g_j:=N^j-\mu_j$. By expanding the definition of the Gowers norm, we obtain
    \begin{align}
    \label{eq:Gow}
    \|g\|_{U^k}^{2^k}=\sum_{B\subseteq \{0,1\}^k}(-1)^{2^k-|B|}\mu_j^{2^k-|B|}N_B
    \end{align}
    where
    \[
N_B:=\underset{x,h_1,\dots,h_k\in\F_p}{\mathbb{E}}\prod_{\omega\in B} N(x+\omega\cdot(h_1,\dots,h_k))^j.
    \]
    For nonempty subset $B\subseteq \{0,1\}^k$, we define the variety
    \[
    V_B:=\{(x,h_1,\dots,h_k,(y_{\omega,i})_{\omega\in B,1\leq i\leq j})\in \F_p^{k+1}\times \F_p^{j|B|} : P(y_{\omega,i})=x+\omega\cdot(h_1,\dots,h_k), \omega\in B,1\leq i\leq j\}.
    \]

    We claim that the dimension of $V_B$ is $k+1$. The projection $\pi : V_B\rightarrow \F_p^{k+1},$ which sends $(x,h_1,\dots,h_k,(y_{\omega,i})_{\omega\in B,1\leq i\leq j})$ to $(x,h_1,\dots,h_k)$ has finite fibers over $\overline{\F}_p$ because each equation $P(y)=t$ has $d$ solutions. More detailed information on this topic can be found in \cite[Chapter 5]{AM}. To verify our claim, note that there are $j|B|$ such equations, so the number of tuples $(y_{\omega,i})_{\omega\in B,1\leq i\leq j}$ satisfying the system of equations is $d^{j|B|}$. Since $V_B$ is a subset of an affine space of dimension $k+1+j|B|$ and is defined by $j|B|$ equations, and because the projection $\pi$ has finite fibers means that its dimension is zero over $\overline{\F}_p$, we have
    \[
    \dim V_B\leq k+1+\dim(\text{fiber of }\pi)=k+1.
    \]
    But we have $\pi$ is surjective, so $\dim V_B\geq k+1$.
    
    Now on one hand, we have that
    \[
    |V_B(\F_p)|=\sum_{x,h_1,\dots,h_k\in\F_p}\prod_{\omega\in B}N(x+\omega\cdot(h_1,\dots,h_k))^j=p^{k+1}N_B
    \]
    by the definition of $N_B.$
    
    On the other hand, using Lang-Weil estimate from Proposition \ref{lem:LangWeil}, we have
    \[
    |V_B(\F_p)|=C_Bp^{k+1}+O_{d,k}(p^{k+1/2})
    \]
    where $C_B$ is the number of absolutely/geometrically irreducible components of $V_B$ of dimension $k+1$. Thus we get 
    \begin{align}
    \label{eq:fibre}
    C_B=N_B+O_{d,k}(p^{-1/2}).
    \end{align}
    The next step is to determine the value of $C_B$.

    To do this, we consider the variety
    \[
    X_j=\{(t,y_1,\dots,y_j)\in\F_p\times \F_p^j : P(y_i)=t,1\leq i\leq j\}.
    \]
    We have that $|X_j(\F_p)|=\displaystyle \sum_{t\in \F_p} N(t)^j=p\mu_j$. By an argument similar to the one used for $V_B$, applying Proposition \ref{lem:LangWeil} again yields
    \[
    |X_j(\F_p)|=C_jp+O_{d,j}(p^{1/2}).
    \]
    This gives us $\mu_j=C_j+O_{d,j}(p^{-1/2}).$ 

    The final relation needed to complete our calculations is an expression for $C_B$ in terms of $\mu_j$, which will allow us to obtain a bound for \eqref{eq:Gow}. 
    
    To achieve this, we need to consider the relation of $V_B$ and $X_j$ geometrically as follows. Let $\displaystyle L : \F_p^{k+1}\rightarrow \F_p^B$ be the linear map defined by $L(x,h_1,\dots,h_k)=(x+\omega\cdot(h_1,\dots,h_k))_{\omega\in B}$. Let $T$ denote the image of $L$. If we let $R=\dim_{\F_p}\text{span}\{(1,\omega) : \omega\in B\}$, we have a linear parameterization $\varphi : \F_p^R\rightarrow T$ such that the equations defining $V_B$ can be rewritten as
    \begin{align}
    \label{eq:polyeq}
    P(y_{\omega,i})=\varphi_{\omega}(u),\quad u\in \F_p^R,\omega\in B,1\leq i\leq j.
    \end{align}

    Here we have that for $\omega\in B$, the map $\varphi_{\omega}$ is defined by $\varphi_{\omega}:=\pi_{\omega}\circ \varphi : \F_p^R\rightarrow \F_p$ where $\pi_{\omega} : T\rightarrow \F_p$ is projection onto the $\omega-$coordinate.
    
    If we let $W_B$ denote the variety defined by \eqref{eq:polyeq}, then $V_B$ is isomorphic to $\F_p^{k+1-R}\times W_B$. To see this, let $v_1,\dots,v_{k+1-R}$ be a basis of $\ker L$ and $u_1,\dots,u_R$ be a basis for the orthogonal complement of $\ker L$ in $\F_p^{k+1}$. Then, writing an element in terms of these bases as $u,v)$ where $u = (u_1,\dots,u_R)$ and $v = (v_1,\dots,v_{k+1-R})$, we have
    \[
    L(u,v)=(\varphi_1(u),\dots,\varphi_{|B|}(u))
    \]
    after re-indexing elements of $B$. Since $\varphi_{\omega}$ is linear in $u_1,\dots,u_R$ and does not depend on $v_1,\dots,v_{k+1-R}$ for any $\omega$, and since $v_1,\dots,v_{k+1-R}$ lie in the kernel of $L$, we have $L(u,v)=L(u,0)$. This means that we can rewrite the equations defining $V_B$ by
    \[
    P(y_{\omega,i})=\varphi_{\omega}(u),\quad \omega\in B,1\leq i\leq j,
    \]
    which is \eqref{eq:polyeq}.
    Therefore, by our definition of $W_B$, we have $V_B\cong \F_p^{k+1-R}\times W_B$.
    
    We also claim that $W_B$ is the pullback of the product $X_j^B:=\displaystyle \prod_{\omega\in B}X_j$ along the linear embedding $\F_p^R\hookrightarrow \F_p^{B}$ where the linear embedding $\phi : \F_p^R\rightarrow \F_p^B$ is given by
    \[
    \phi(u)=(\varphi_{\omega}(u))_{\omega\in B}.
    \]
    We have that a point in $X_j^B$ is a tuple $((t_{\omega},y_{\omega,1},\dots,y_{\omega,j}))_{\omega\in B}$ satisfying $P(y_{\omega,i})=t_{\omega}$ for any $i=1,\dots,j$. Let $\pi_B : X_j^B :\rightarrow \F_p^B$ denote the projection onto the $t_{\omega}$ coordinates. 
    We claim that $W_B\cong X_j^B\times_{\F_p^B}\F_p^R:=\{((t_{\omega},y_{\omega,i})_{\omega\in B},u)\in X_j^B\times \F_p^R : t_{\omega}=\varphi_{\omega}(u)\}$. Note that the conditions
    \[
    t_{\omega}=\varphi_{\omega}(u)
    \]
    determine each $t_{\omega}$ uniquely from $u.$ So we have the projection
    \[
    ((t_{\omega},y_{\omega,i})_{\omega\in B},u)\mapsto (u,(y_{\omega,i})_{\omega\in B})
    \]
    is an isomorphism onto its image. The remaining equations defining $X_j^B$ become
    \[
    P(y_{\omega,i})=t_{\omega}=\varphi_{\omega}(u),
    \]
    which is exactly $W_B$. So we get that $W_B$ is the claimed pullback.

    Coming back to our calculations, we have the forms $\omega\mapsto (1,\omega)$ are pairwise distinct. This means that the linear section is generic on the dense open locus where the fibres of $X_j\rightarrow \F_p$ are \'{e}tale by the fact that the derivative of $P$ is not identically zero. This gives us $N_B=C_j^{|B|}$.

    Recalling \eqref{eq:fibre}, we can now conclude that
    \[
    C_B=C_j^{|B|}+O_{d,k}(p^{-1/2})=\mu_j^{|B|}+O_{d,k}(p^{-1/2}).
    \]
    Substituting what we have to \eqref{eq:Gow}, we get
    \begin{align*}
    \|g\|_{U^k}^{2^k}&=\sum_{B\subseteq\{0,1\}^k}(-1)^{2^k-|B|}\mu_j^{2^k-|B|}\left(\mu_j^{|B|}+O_{d,k}(p^{-1/2})\right)\\
    &=\mu_j^{2^k}\sum_{B\subseteq \{0,1\}^k}((-1)^{2^k-|B|})+O_{d,k}(p^{-1/2})=O_{d,k}(p^{-1/2}).
    \end{align*}
    Taking the $2^k$-th root, by our earlier observation that $\|f\|_{U^k}\ll_d \|g\|_{U^k}$, we end up with
    \[
    \|f\|_{U^k}=O_{d,k}(p^{-1/2^{k+1}}).
    \]
\end{proof}
\begin{remark}
    The proof above works for any nondegenerate polynomial $P$ and $k\geq 3$, so in particular in the case when $k\geq 3$, we only need $d<p$.
\end{remark}

\section*{Acknowledgement}
The author would like to thank Fernando Xuancheng Shao for his advice and suggestions. The author also thanks Borys Kuca for his ideas and suggestions. The author is partially supported by NSF grant DMS-2452462.
\section*{AI Disclosure}
The paper is written by human. Gemini 3.5 Flash was used for copyediting assistance during the final edits of this paper. The author used Chat GPT 5.5 Thinking Mode to find \cite{AM} as reference for dimension of variety. All mathematical arguments and results in this paper were human-generated.

\end{document}